\documentclass{article}

\usepackage[]{datetime}

\usepackage{cite}
\usepackage{amssymb} \usepackage{graphicx}
\usepackage[utf8x]{inputenc}
\usepackage{enumerate} \usepackage{amsfonts} \usepackage{amsmath}
\usepackage{stmaryrd} 
\usepackage{xspace}  

\usepackage{hyperref}
\usepackage{amsthm}
\usepackage{color}

\newcommand{\ie}{\emph{i.e.}}

\newcommand{\floor}[1]{\left\lfloor#1\right\rfloor}
\newcommand{\ceil}[1]{\left\lceil#1\right\rceil}
 
\newcommand{\paren}[1]{\left(#1\right)} 
\newcommand{\acc}[1]{\left\{ #1 \right\}}

\newcommand{\llam}[1]{{\Lambda(#1)}}
\newcommand{\N}{\mathbb{N}}

\newcommand{\intv}[2]{\left \llbracket #1, #2 \right \rrbracket}
\newcommand{\tw}{\mathbf{tw}}
\newcommand{\pw}{\mathbf{pw}}

\newcommand{\vertices}[1]{{V(#1)}}
\newcommand{\edges}[1]{{E(#1)}}
\newcommand{\leaves}[1]{L(#1)}

\newcommand{\induced}[2]{#1\left [#2 \right ]} 

\newcommand{\card}[1]{\left | #1 \right |} 
 
\newcommand{\lminor}{\leq_\mathrm{m}} 
\newcommand{\wheel}[1]{\mathrm{W}_{#1}} 
\newcommand{\dwheel}[1]{\mathrm{W}^2_{#1}} 
\newcommand{\sg}[1]{\mathrm{\Xi}_{#1}} 

\newcommand{\half}{\frac{1}{2}}

\newcommand{\lca}[1]{\mathbf{lca}_{#1}} 

\newcommand{\roott}{\mathbf{root}}
\DeclareMathOperator{\diam}{{\bf diam}} 

\newcommand{\patht}[3]{{#1 #2 #3}}
\newcommand{\pathtend}[3]{{#1 #2 \mathring{#3}}}
\newcommand{\pathtbeg}[3]{{\mathring{#1} #2 #3}}

\renewcommand{\leq}{\leqslant} \renewcommand{\geq}{\geqslant}

\newtheorem{theorem}{Theorem}
\newtheorem{lemma}{Lemma}
\newtheorem{corollary}{Corollary}
\newtheorem{proposition}{Proposition}
\newtheorem{claimb}{Claim}

\theoremstyle{remark}
\newtheorem{remark}{Remark}

\theoremstyle{definition}
\newtheorem{definition}{Definition}

\def\claimb{$$\vcenter\bgroup\advance\hsize by -8em\noindent
\refstepcounter{claimb}\ignorespaces\it} \makeatletter
\def\endclaimb{\rm\egroup\leqno(\theclaimb)$$\global\@ignoretrue}
\makeatother

\usepackage{tikz}
\usetikzlibrary{decorations.pathreplacing}
\tikzstyle{every node} = [draw, circle, fill = black, minimum size = 4pt, inner sep = 0pt]
\tikzstyle{normal} = [draw=none, fill = none]

\begin{document}

\title{Low Polynomial Exclusion of Planar Graph Patterns\thanks{The second author has been 
co-financed by the E.U. (European Social Fund - ESF) and Greek national funds through the Operational Program ``Education and Lifelong Learning'' of the National Strategic Reference Framework (NSRF) - Research Funding Program: ``Thales. Investing in knowledge society through the European Social Fund''.
{Emails:
\href{mailto:jean-florent.raymond@ens-lyon.org}{{\sf
jean-florent.raymond@mimuw.edu.pl}},  
\href{mailto:sedthilk@thilikos.info}{{\sf
sedthilk@thilikos.info}}}}}

\author{Jean-Florent Raymond~\thanks{AlGCo project team, LIRMM, Montpellier,
France.}\ \thanks{Faculty of Mathematics, Informatics and
Mechanics, University of Warsaw, Poland, and University of Montpellier, France.}\and Dimitrios
M. Thilikos$^\ddag$\thanks{Department of Mathematics,
National and Kapodistrian University of Athens}}

\date{}

\maketitle

\begin{abstract}
\noindent The celebrated grid exclusion theorem states that for every $h$-vertex planar graph $H$, there is 
a constant $c_{h}$ such that if a graph $G$ does not contain $H$ as a minor then $G$ has treewidth at most $c_{h}$.
We are looking for patterns of $H$ where this bound can become a low degree polynomial. We provide 
such bounds for the following parameterized graphs: the wheel ($c_{h}=O(h)$),  the double wheel  ($c_{h}=O(h^2\cdot \log^{2} h)$),
any graph of pathwidth at most $2$ ($c_{h}=O(h^{2})$), and the yurt graph  ($c_{h}=O(h^{4})$).
\end{abstract}


\section{Introduction}

Treewidth is one of the most important graph invariants in modern graph theory.
It has been introduced in~\cite{RobertsonS-II} by Robertson and Seymour 
as one of the cornerstones of their Graph Minors series. Apart from its huge
combinatorial value, it has been extensively used 
in graph algorithm design (see~\cite{Bodlaender98apa} for an extensive survey on treewidth).
On an intuitive level, treewidth expresses how close the topology of
the graph is to that of a tree and, in a sense, can be seen as a
measure of the ``global connectivity'' of a~graph.

Formally,  a \emph{tree decomposition} of a graph~$G$
is a pair~$(T,\mathcal{X})$ where $T$ is a tree and
$\mathcal{X}$ a family $(X_t)_{t \in \vertices{{T}}}$ of
subsets of $\vertices{G}$ (called \emph{bags}) indexed by
elements of $\vertices{T}$ and such that
 \begin{enumerate}[(i)]
 \item $\bigcup_{t \in \vertices{{T}}} X_t = \vertices{G}$;
 \item for every edge~$e$ of~$G$ there is an element of~$\mathcal{X}$
containing both ends of~$e$;
 \item for every~$v \in \vertices{G}$, the subgraph of~${T}$
induced by $\{t \in \vertices{{T}}\mid {v \in X_t}\}$ is connected.
 \end{enumerate}
The \emph{width} of a tree decomposition is 
equal to $\max_{t \in \vertices{{T}}}~{\card{X_t} - 1}$, while the
\emph{treewidth} of~$G$, written~$\tw(G)$, is the minimum width of any
of its tree decompositions. Similarly one may define the notions of path decomposition 
and pathwidth by additionally asking that $T$ is a path (see Section~\ref{kg4rfs}).

%
%
%

We say that a graph $H$ is a {\em minor} of a graph $G$ if a graph isomorphic to 
$H$ can be obtained from a subgraph of $G$ by applying a series of edge contractions,
and we denote this fact by $H\lminor G$.

\paragraph{The grid exclusion theorem.}
One of the most celebrated results from the Graph Minors series of Robertson and Seymour is the following result, also known as the {\sl grid exclusion theorem}.

\begin{proposition}[\!\!\cite{RobertsonS86GMV}]
\label{thlsk}
There exists a function $f: \N\rightarrow \N$
such that, for  every for every planar graph $H$ on $h$ vertices,
every graph $G$ that does not contain a minor isomorphic to $H$ has treewidth 
at most $f(h)$.
\end{proposition}

The original proof the the above result in~\cite{RobertsonS86GMV} did not provided any explicit 
estimation for the function $f$. Later,
in~\cite{RobertsonST94}, Robertson, Seymour, and Thomas 
proved the same result for $f(h)=2^{O(h^{5})}$, while a less complicated 
proof appeared in~\cite{DiestelJGT99high}. The bound $f(h) \leq h-2$
was also obtained in~\cite{Bie} in the case where $H$ is required to be
a forest.

 For a long time,
whether Proposition~\ref{thlsk} can be proved for a polynomial $f$ was an open problem. 
In~\cite{RobertsonST94}, an $\Omega(h^2\cdot \log h)$ lower bound 
was provided for the best possible estimation of $f$ and was also conjectured that 
the optimal estimation should not be far away from this lower bound. In fact, a more 
precise variant of the same conjecture was given by Demaine, Hajiaghayi, and Kawarabayashi in~\cite{DemaineHK09}
where they conjectured that Proposition~\ref{thlsk} holds for $f(h)=O(h^{3})$. The bounds of~\cite{RobertsonST94}
were recently improved by Kawarabayashi and Kobayashi \cite{kawarabayashiK2line}, where 
they show that  Proposition~\ref{thlsk} holds for $f(h)=2^{O(h\cdot \log h)}$. The same bounds were obtained by Leaf and Seymour \cite{LeafS12sube}. Until recently, this was the best known estimation of the function~$f$.

Very recently,  in a breakthrough result~\cite{ChekuriC13poly}, Chekuri and Chuzhoy  proved that 
Proposition~\ref{thlsk} holds for $f(h)=O(h^{228})$.
The remaining open question is whether the degree of this polynomial bound can be substantially 
reduced in general. In this direction,
one may still consider restrictions 
either on the graph $G$ or on the graph $H$ that yield a low polynomial dependence 
between the treewidth and the size of the excluded minor.
In the first direction, Demaine and Hajiaghayi proved in~\cite{DemaineH08line} that,
when $G$ is restricted to belong in some graph class excluding some 
fixed graph $R$ as a minor, then Proposition~\ref{thlsk}  (optimally) holds 
for $f(h)=O(h)$. Similar results have been proved by Fomin, Saurabh, and Lokshtanov, in~\cite{FominLS12bidi}, for the case where $G$ is either a unit disk graph or a map graph that does not contain a clique as a subgraph.

In a second direction, one may consider $H$ to be some specific planar graph and find a good upper bound 
for the treewidth of the graphs that exclude it as a minor. More generally, we can consider a parametrized 
class of planar graphs ${\cal H}_{k}$ where each graph in ${\cal H}_{k}$ has 
size  bounded by a polynomial in $k$, and prove that the following fragment of Proposition~\ref{thlsk} holds for some low degree polynomial function $f:\N\rightarrow \N$:
\begin{eqnarray}
\mbox{$\forall k\geq 0\  \forall {H\in{\cal H}_{k}},\ if H\not\lminor G$ then ${\bf tw}(G)\leq f(k)$.}\label{sgdb}
\end{eqnarray}
The  question can be stated as follows: find pairs $({\cal H}_{k}, g(k))$
for which~\eqref{sgdb} holds for some $f(k) = O(g(k)$, where 
${\cal H}_{k}$ is as general as possible and $g$ is as small as possible (and certainly polynomial).
It is known, for example, that~\eqref{sgdb} holds for the pair 
$(\{C_{k}\},k)$, where $C_{k}$ is the cycle or a path of $k$ vertices   (see e.g.~\cite{Bodl93a,FellowsL94}),
and for the pair $(\{K_{2,k}\},k)$, (see~\cite{BodlaenderLTT97onin}).
Two more results in the same direction that appeared recently are the following:
according to the result of  Birmelé,  Bondy, and 
Reed in~\cite{BirmeleBR07bram},~\eqref{sgdb} holds for the pair $({\cal P}_{k}, k^2)$ where 
${\cal P}_{k}$ contains all minors of $K_{2}\times C_{k}$ (we denote by $K_{2}\times C_{k}$ the Cartesian product of  $K_{2}$
and the cycle of $k$ vertices, also known as the {\em $k$-prism}).
Finally, one of the consequences of the recent results of Leaf and Seymour 
in~\cite{LeafS12sube}, implies that~\eqref{sgdb} holds for
the pair $({\cal F}_{r}, k)$, where  ${\cal F}_{r}$ contains every graph on $r$ vertices 
that contains a vertex that meets all its cycles.

\paragraph{Our results.} In this paper we provide polynomially bounded minor exclusion theorems  
for  the following parameterized graph classes:
\begin{itemize}
\item[${\cal H}_{k}^0$:] containing all graphs on $k$ vertices that have pathwidth at most 2.
\item[${\cal H}_{k}^1$:] containing all minors of a wheel on $k+1$ vertices -- see Figure~\ref{fig:w6}.

\item[${\cal H}_{k}^2$:] containing all minors of  a double wheel on $k+2$ vertices  -- see Figure~\ref{fig:w6}.
\item[${\cal H}_{k}^3$:] containing all minors of  the yurt graph on $2k+1$ vertices (i.e. the graph 
obtained it we take a $(2\times k)$-grid and add a new vertex adjacent with all the vertices 
of its ``upper layer'' -- see Figure~\ref{fig:y}).
\end{itemize}
Notice that none of the above classes is minor comparable with the classes ${\cal P}_{k}$ and ${\cal F}_{k}$
treated in~\cite{BirmeleBR07bram} and~\cite{LeafS12sube}.
Moreover, ${\cal H}_{k}^1\subset {\cal H}_{k}^2 \subset {\cal H}_{k}^3$, while ${\cal H}_{k}^0$ is not minor comparable with the other three. In this paper we prove that~\eqref{sgdb} holds for the pairs:\medskip
\begin{itemize}
\item $({\cal H}_{k}^0, k^2)$, 
\item  $({\cal H}_{k}^1, k)$, 
\item  $({\cal H}_{k}^2, k^2\log^{2} n)$, and 
\item  $({\cal H}_{k}^3, k^4)$.
\end{itemize}
The above results are presented in detail, without the $O$-notation, in Section~\ref{sec:res}.
All of our proofs use as a departure point the results of Leaf and Seymour in~\cite{LeafS12sube}.

\section{Definitions}
\label{kg4rfs}

All graphs in this paper are finite and simple, \ie, do not have loops
nor multiple edges. We denote by $\vertices{G}$ (resp. $E(G)$)
the sets of vertices (resp. edges) of $G$. 
For every $i,j\in \N$, $i\leq j$, the notation $\intv{i}{j}$ stands for the interval
of integers $\{i, i+1, \dots, j\}$.
Logarithms are binary.

\begin{definition}[path decomposition, pathwidth]
  A \emph{path decomposition} of a graph~$G$ is a tree
  decomposition~${T}$ of~$G$ such that~${T}$ is a
  path. Its width is the width of the tree decomposition~${T}$
  and the \emph{pathwidth} of~$G$, written~$\pw(G)$, is the minimum
  width of any of its path decompositions.
  An \emph{optimal path decomposition} is a path decomposition of
  minimum width.
\end{definition}

\begin{definition}[contraction and dissolution]
The \emph{contraction} of an edge $\{u,v\}$ in a graph $G$ is the operation
which creates a new vertex adjacent to the neighbors of $u$ and those
of $v$, and deletes both $u$ and $v$. The \emph{dissolution} of a
vertex of degree two is the contraction of one of the edges incident
with~it.
\end{definition}

\begin{definition}[minor model]
  A \emph{minor model} (sometimes abbreviated \emph{model}) of a graph
  $H$ in a graph $G$ is a pair $({\cal M},\varphi)$ 
  where ${\cal M}$ is a set of pairwise disjoint subsets
of $\vertices{G}$ such that $\forall X \in {\cal M}$, $\induced{G}{X}$
is connected and $\varphi \colon \vertices{H} \to {\cal M}$ is a
bijection that satisfies $\forall \{u,v\} \in \edges{H}, \exists u'
\in \varphi(u), \exists v' \in \varphi(v),\ \{u',v'\} \in
\edges{G}$. We say that a graph $H$ is a \emph{minor} of a graph $G$
($H \lminor G$) if there is a minor model of $H$ in $G$. Notice that
$H$ is a minor of $G$ if $H$ can be obtained from subgraph of~$G$
by edges~contractions.
\end{definition}

\begin{definition}[linked set]
Let $G$ be a graph and $S \subseteq \vertices{G}$. The set $S$ is said to be
\emph{linked} in $G$ if for every two subsets~$X_1$, $X_2$ of~$S$ (not
necessarily disjoint) such that $\card{X_1} = \card{X_2}$, there is
a set~$Q$ of~$\card{X_{1}}$ {(vertex-)}disjoint paths between~$X_1$ and~$X_2$ in~$G$ whose
length is not one (but can be null) and whose endpoints only are in~$S$.
\end{definition}
\begin{definition}[separation]
  A pair~$(A,B)$ of subsets of~$\vertices{G}$ is a called a
  \emph{separation of order $k$ in~$G$} if $k = \card{A \cap B}$, none
  of $A,B$ is a subset of the other, and
  there is no edge of $G$ between $A \setminus B$ and $B \setminus A.$
\end{definition}

\begin{definition}[left-contains, \cite{LeafS12sube}]
  Let $H$ be a graph on $r$ vertices, $G$ a graph and $(A,B)$ a
  separation of order $r$ in $G$. We say that $(A,B)$
  \emph{left-contains} $H$ if $\induced{G}{A}$ contains a minor model ${\cal M}$ of $H$
  such that $\forall M \in {\cal M},\ {\card{M \cap (A \cap B)} = 1}$
\end{definition}

\begin{definition}[Trees and cycles]
Given a tree $T$ we denote by $\leaves{T}$ the set of its leaves, \ie~vertices of degree 1
and by $\diam(T)$ its diameter, that is the maximum length (in number of edges) of a path
in $T.$

For every two vertices $u,v \in \vertices{T},$
there is exactly one path in $T$ between $u$ and $v$, that we denote
by $\patht{u}{T}{v}$. Also, given that $\patht{u}{T}{v}$ has at least 2
vertices, we denote by $\pathtbeg{u}{T}{v}$
(resp. $\pathtend{u}{T}{v}$) the path $\patht{u}{T}{v}$ with the vertex
$u$ (resp. $v$) deleted.

Let $C$ be a cycle on which we fixed some orientation.
Then, there is exactly one path following this orientation between any two
vertices $u,v \in \vertices{C}$. Similarly, we denote this path by
$\patht{u}{C}{v}$ and we define $\pathtbeg{u}{C}{v}$ and
$\pathtend{u}{C}{v}$ as we did for the tree.

In a rooted tree $T$ 
with root $r$, the \emph{least common ancestor} of two vertices $u$ and
$v$, written $\lca{T}(u,v)$ is the first common vertex of the paths
$\patht{u}{T}{r}$ and $\patht{v}{T}{r}$. We refer to the root of $T$
by the notation $\roott(T).$

For every integer $h>0$, we denote by $B_h$ the complete binary tree
of height~$h.$
\end{definition}

\section{Results}
\label{sec:res}

We present in this paper bounds on the treewidth of graphs not containing the
following parameterized graphs: the wheel of order $k$
(section~\ref{sec:excl_wheel}), the double wheel of order $k$
(section~\ref{sec:excl_dwheel}), any graph on $k$ vertices and
pathwidth at most 2 (section~\ref{sec:excl_sq}) and the yurt graph
of order $k$ (section~\ref{sec:excl_yurt}). The definitions of these 
graphs can be found in the corresponding sections. In section
\ref{sec:pre}, we recall some propositions that we will use and we
prove two lemmata which will be useful later. The theorems we then
prove are the following.

 \begin{theorem}
Let $k>0$ be an integer and $G$ be a graph. If $\tw(G) \geq 36k -
2$, then $G$ contains a wheel of order $k$ as minor.
 \end{theorem}

 \begin{theorem}
   Let $k>0$ be an integer and $G$ be a graph. If
   $\tw(G) \geq 12 (8k \log(8k) + 2)^2 -4$, then $G$ contains a double
   wheel of order at least $k$ as minor.
 \end{theorem}

\begin{theorem}
  Let $k>0$ be an integer, $G$ be a graph and $H$ be a graph on $k$
  vertices and of pathwidth at most 2. If $\tw(G) \geq 3k(k-4) +8$ then $G$
  contains $H$ as minor.
\end{theorem}

\begin{theorem}
  Let $k>0$ be an integer and $G$ be a graph. If $\tw(G) \geq 6 k^4 -
  24 k^3 + 48 k^2 - 48 k + 23$, then $G$ contains the yurt graph of
  order $k$ as minor.
\end{theorem}

\section{Preliminaries}
\label{sec:pre}

\begin{proposition}[\hspace{-.01mm}\protect{\cite[(4.3)]{LeafS12sube}}]
\label{p:leaf}
  Let~$k > 0$ be an integer, let~$F$ be a forest on~$k$ vertices and let~$G$ be a graph. If
  $\tw(G) \geq \frac{3}{2} k - 1,$
  then~$G$ has a separation~$(A,B)$ of order~$k$ such that
  \begin{itemize}
  \item $\induced{G}{B \setminus A}$ is connected;
  \item $A \cap B$ is linked in~$\induced{G}{B}$;
  \item $(A,B)$ left-contains $F$.
  \end{itemize}
\end{proposition}


\begin{proposition}[Erdős--Szekeres Theorem,
\cite{ErdosSzekeres}\label{p:es}]
Let~$k$ and~$\ell$ be two positive integers. Then any sequence
of~${(\ell-1)(k-1) + 1}$ distinct integers contains either an
increasing subsequence of length~$k$ or a decreasing subsequence of
length~$\ell$.
\end{proposition}

\begin{lemma}\label{l:diamleaves}
  For every tree $T$, $\card{\vertices{T}} \leq \frac{\card{\leaves{T}} \cdot \diam(T)}{2} + 1$.
 \end{lemma}
 
 \begin{proof}
Root $T$ to a vertex $r\in \vertices{T}$ that is halfway of a longest path of $T$.
For each leaf $x\in \leaves{T}$, we know that
$\card{\vertices{\pathtend{x}{T}{r}}}\leq \floor{\frac{\diam(T)}{2}}$. Observe that
$\vertices{T} = \{r\} \cup \bigcup_{x\in \leaves{T}}
\vertices{\pathtend{x}{T}{r}}$. Therefore,
\begin{align*}
\card{\vertices{T}} &\leq \sum_{x\in \leaves{T}}
\card{\vertices{\pathtend{x}{T}{r}}} + 1\\
\card{\vertices{T}} & \leq \card{\leaves{T}}\cdot
\floor{\frac{\diam(T)}{2}} + 1.
\end{align*}
Notice that equality holds for the subdivided star (obtained from
$K_{1,n}$ by subdividing $k$ times every edge, for some $n,k \in \N$).
 \end{proof}
 
  \begin{definition}[The set $\llam{T}$]\label{d:llam}
    Let $T$ be a tree. We
    denote by $\llam{T}$ the set containing every graph obtained as follows: take the 
    disjoint union of $T$, a path $P$ where $\card{\vertices{P}} \geq \sqrt{\card{\leaves{T}}}$, and an
    extra vertex $v_{\rm new}$, and add edges such that 
    \begin{enumerate}[(i)]
    \item there is an edge between $v_{\rm new}$ and every vertex of $P;$ \label{e:it1}
    \item there are $\card{\vertices{P}}$ disjoint edges between $P$ and $\leaves{T};$ \label{e:it2}
    \item there are no more edges than the edges of $T$ and $P$ and the
      edges mentioned in (\ref{e:it1}) and (\ref{e:it2}).
    \end{enumerate}
  \end{definition}

  \begin{lemma}\label{l:zeta}
    Let $n\geq 1$ be an integer, $T$ be a tree on $n$ vertices an let $G$
    be a graph. If $\tw(G) \geq 3 n - 1$, then $H\lminor G$ for some
    $H \in \llam{T}.$
  \end{lemma}

  \begin{proof}
    Let $n$, $T$, and $G$ be as in the statement of the lemma. Let $l$
    be the number of leaves of $T$, and let $J$ be a path on $l$
    vertices. We consider the disjoint union of $J$ and $T$.
    
    The graph $G$ has treewidth at least $\frac{3}{2}(n+l) - 1$,
    then by Proposition~\ref{p:leaf}, $G$ has a separation $(A,B)$
    of order~$n+l$ such that
    \begin{enumerate}[(i)]
    \item \label{e:co} $\induced{G}{B \setminus A}$ is connected;
    \item \label{e:linked} $A \cap B$ is linked in~$\induced{G}{B};$
    \item \label{e:lc} $(A,B)$ left-contains the graph $J \cup T$.
    \end{enumerate}
    
    Let $({\cal M},\varphi)$ be the a model of $J \cup T$ in
    $\induced{G}{A}$ that witnesses~(\ref{e:lc}).

    We call the vertices of $A \cap B$ that belong to $\varphi(v)$ for
    some $v \in \vertices{J}$ the $J$-part, and
    vertices that belong to $\varphi(v)$ for
    some $v \in \leaves{T}$ forms the $\leaves{T}$-part. Notice that
    two distinct vertices of the $J$-part (resp. $\leaves{T}$-part)
    will be contracted to distinct vertices by the model.

    Let $\mathcal{P}$ a set of $l$ disjoint paths with the one endpoint in
    the $J$-part and the other in the $\leaves{T}$-part, and whose interior
    belongs to $B \setminus A$. The existence of such paths is given
    by~(\ref{e:linked}). For each $P \in \mathcal{P}$, we arbitrarily
    choose a vertex $v_P$ of the interior of $P$, that is, $v_P \in
    \vertices{P} \setminus A$. By~(\ref{e:co}), $\induced{G}{B \setminus
      A}$ is connected: let $Y$ be a smallest tree spanning the
    vertices $\{v_P\}_{P \in \mathcal{P}}$. Let $s =
    \sqrt{\card{\leaves{T}}}$, and let $Y^*$ be the tree obtained from $Y$ by dissolving every vertex
    of degree two that is not $v_P$ for some $P \in \mathcal{P}$.
    We are now facing two possible situations.
    \smallskip

    \noindent \textit{Case 1:} $Y^*$ has a path of length $s$.
    Let $Q$ be the path of $Y$ corresponding to a path of lenght $s$
    in $Y^*$ and let $S$ be the set of vertices $u \in \vertices Q$ that are not of
    degree two or that are $v_P$ for some $P \in \mathcal{P}$.
    Observe that from every $u \in S$, there is a
    path $J_u$ to the $L(T)$-part and a path $J'_u$ to the $J$-part. Indeed, if
    $u=v_P$ for some $P \in \mathcal{P}$, then $u$ is a vertex of $P$
    linking (by definition) a vertex of the $L(T)$-part to a vertex of
    the $J$-part. Otherwise, $u$ is of degree at least 3 in $Y$ and
    every leaf of the subtrees of $Y \setminus Q$ (at least one of which is adjacent
    to $u$), is a $v_P$ for some $P \in \mathcal{P}$ (by minimality of $Y$), so is connected to the
    $L(T)$-part and the $J$-part as explained above.
    Furthermore, for every two distinct $u,v \in S$, the
    aforementioned path are disjoint.

    Let us now summarize. $\induced{G}{\cup_{v \in \vertices{J}} \varphi(v)}$ is a connected
    subgraph of $G$, which is connected by the $s$ disjoint paths
    ${J'_u}_{u \in S}$ to the path $Y$. All the endpoints of the paths
    ${J'_u}_{u \in S}$ on $Y$ are connected by $s$ disjoint paths
    ${J_u}_{u \in S}$ to the $L(T)$-part, which correspond to the
    leaves in a model of~$T$. Therefore this graph contains a member
    of $\llam{T}$ as a minor, as required.

    \noindent \textit{Case 2:} $\diam(Y^*) < s $.
    From Lemma~\ref{l:diamleaves}, $\card{\leaves{Y}} = \card{\leaves{Y^*}}  \geq s$.
    Observe that $\leaves{Y}\subseteq \{v_P\}_{P \in \mathcal{P}}$ (this
    follows by the minimality of $Y$).
    Let $S = \vertices{Y} \setminus \leaves{Y}$. We consider the minor
    of $G$ obtained by contracting, for every $P \in \mathcal{P}$ such
    that $v_P \in \leaves{Y}$, every edge of the subpath connecting
    the $J$-part to a leaf of $Y$.
    In this graph, $S$ induces a connected subgraph adjacent to at
    least $s$ distinct vertices of the $J$-part. All these $s$
    vertices of the $J$-part are connected by $s$ disjoint paths to
    distinct vertices of the $\leaves{T}$-part. 
    Thus this contains a member
    of $\llam{T}$ as a minor, and so do $G$.
    
  \end{proof}

\section{Excluding a wheel with a linear bound on treewidth}
\label{sec:excl_wheel}

\begin{definition}[wheel]
  Let $r>2$ be an integer. The \emph{wheel} of order $r$ (denoted
  $W_r$) is a cycle of length $r$ whose each vertex is adjacent to an
  extra vertex, in other words it is a the graph of the form

  \begin{align*}
    \vertices{G} & = \{o, w_1,\dots, w_r\}\\
    \edges{G} & = \{\acc{w_1,w_2}, \acc{w_2,w_3}, \dots, \acc{w_{r-1}, w_r}, \acc{w_r, w_1}\} \cup \{\acc{o,w_1}, \dots, \acc{o,w_r}\}
  \end{align*}
(see Figure~\ref{fig:w6} for an example).
\end{definition}

\begin{figure}[h]
  \centering
  \begin{tikzpicture}
    \foreach \t in {1, ..., 6} {
      \draw (\t * 60 :1) -- (\t * 60 + 60:1);
      \draw (\t * 60 :1) node {} -- (0,0);
      \node[normal] at (-\t * 60 + 180:1.5) {$w_{\t}$};
    }
    \node (0,0) {};
    \node[normal] at (30:.5) {$o$};
      
    \begin{scope}[xshift = 5cm]
          \foreach \t in {1, ..., 6} {
      \draw (\t * 60 :1) -- (\t * 60 + 60:1);
      \draw (\t * 60 :1) node {} -- (0,0);
      \node[normal] at (-\t * 60 + 180:1.5) {$w_{\t}$};
    }

    \foreach \t in {1, ..., 3} {
      \draw (\t * 60 :1) .. controls +( 30*\t :\t*\t/4) and +(200 - \t*30 :1) .. (30:3);
      \draw (60-\t * 60 :1) .. controls +(60- 30*\t:\t*\t/4) and +(200 + \t*30 :1) .. (30:3);
    }

    \draw (0,0) node {};
    \draw (30:3) node {};

    \node[normal] at (30:.5) {$o_1$};
    \node[normal] at (30:3.25) {$o_2$};
    \end{scope}

  \end{tikzpicture}
  \caption{A wheel of order six (left) and a double wheel of order 6 (right)}
  \label{fig:w6}
\end{figure}
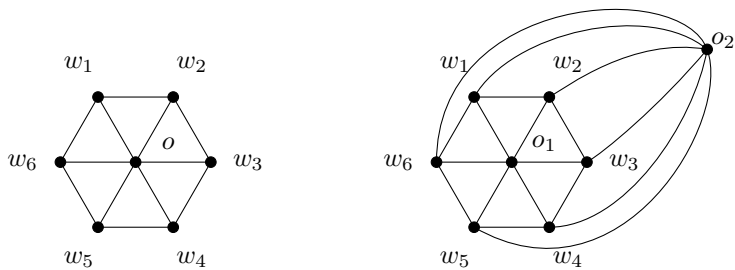

 \begin{lemma}\label{l:wtp}
   Let $h > 2$ be an integer. Let $G$ be a graph obtained from the
   union of the tree $T = B_h$ and a path $P$ by adding the edges
   ${\acc{l , \psi(l)} \in \edges{G}}$ for every $l \in \leaves{T}$,
   where $\psi \colon \leaves{T} \to \vertices{P}$ is a bijection. Then $G$
   contains a wheel of order~$2^{h-2} + 1$ as a minor.
 \end{lemma}

 \begin{proof}
   Let $h$, $\psi$, $T$, $P = p_1\dots p_{2^h}$ and $G$ be as
   above. Let $r$ be the root of $T$.

   In the arguments to follow, if $t \in \vertices{T}$, we denote by $T_t$ the subtree of
   $T$ rooted at $t$ (\ie~the subtree of $T$ whose vertices are all the vertices $t' \in
   \vertices{T}$ such that the path $\patht{t'}{T}{r}$ contains $t$).

   We consider the vertices 
   $u = \psi^{-1}(p_1) \in \leaves{T}$ and $v  
   = \psi^{-1}(p_{2^h}) \in \leaves{T}$ and $w = \lca{T}(u,v) \in
   \vertices{T} \setminus \leaves{T}$.

   Let $\tau$ be a largest subtree of $T$ which is
   disjoint from $\patht{u}{T}{v}$.  Let $L_\tau = \leaves{\tau} \cap
   \leaves{T}$ and let~$Q = \psi(L_\tau) \subseteq P$.
   It is not hard to see that $G$ contains $W_{\card{Q} + 1}$ as a
   minor. Indeed, the paths $P$ and $\patht{u}{T}{v}$ together with the edges
   $\{p_1,u\}$ and $\{p_{2h}, v\}$ form a cycle in $G$. Besides, the
   tree $\tau$, which is disjoint from this cycle, has at least
   $\card{Q} +1$ vertices that are adjacent to distinct vertices of
   $P$: $\card{Q}$ of them are the elements of $Q$, and the other one is the
   (only) vertex of $\tau$ adjacent to $\patht{u}{T}{v}$ (which exists
   by maximality of $\tau$). In
   the subgraph of $G$ induced by $\vertices{P} \cup
   \vertices{\patht{u}{T}{v}} \cup \vertices{\tau}$, contracting
   $\tau$ to a vertex gives a vertex adjacent to at least $\card{Q}+1$
   vertices of a (non necessarily induced) cycle, a graph containing
   $W_{\card{Q}+1}$ as subgraph.

   Depending on $G$, $\card{Q}$ may take different values. However, we
   show that it is never less than $2^{h-2}$. Remember, $\card{Q}$ is
   the number of leaves that a largest subtree of $T$ that is
   disjoint from $\patht{u}{T}{v}$ shares with $T$. The root $r$ of
   $T$ has two children $r_1$ and $r_2$, inducing two subtrees $T_{r_1}$ and
   $T_{r_2}$ of $T$. Recall, $w = \lca{T}(u,v).$  \smallskip

   \noindent \emph{Case 1.} $w \neq r$.
   As $w \neq r$, $w$ is a vertex of one of $\{T_{r_1}, T_{r_2}\}$, say
   $T_{r_1}$, which contains also $u$ and $v$, and thus the path
   $\patht{u}{T}{v}$. The other subtree $T_{r_2}$ is then disjoint from
   $\patht{u}{T}{v}$, it has height $h-1$ and is complete so it has
   $2^{h-1}$ leaves. Consequently, in this case $\card{Q} \geq 2^{h-1}.$\smallskip

   \noindent \emph{Case 2.} $w = r$.
   In this case, the path $\patht{u}{T}{v}$ contains $r$ (and $r\neq
   u$, $r \neq v$ as $u$ and $v$ are leaves) so $u$ and $v$ are not in
   the same subtree of $\{T_{r_1}, T_{r_2}\}$ and
   $\patht{u}{T}{v}$ contains the two edges $\{r, r_1\}$ and
   $\{r,r_2\}$. For every $i \in \{1,2\}$, we denote by $r_{i,1}$ and 
   $r_{i,2}$
   the two children of $r_i$ in $T$. We assume without loss of generality
   that $u \in \vertices{T_{r_{1,1}}}$ and $v \in \vertices{T_{r_{2,1}}}$
   (if not, we just rename the $r_i$'s ans $r_{i,j}$'s). Notice that the
   path $\patht{u}{T}{v}$ is the concatenation of the paths
   $uT_{r_1}r_1$, $r_1Tr_2$, $r_2T_{r_2}v$. Since the tree $T_{r_{1,2}}$
   is disjoint from $\patht{u}{T}{v}$, is complete and is of height
   $h-2$, it has $2^{h-2}$ leaves. Therefore we have $\card{Q} \geq
   2^{h-2}.$\smallskip

In both cases, $\card{Q} \geq 2^{h-2}$ and according to what we proved
before, $G$ contains a model of $\wheel{\card{Q}+1}$. As every
wheel contains as a minor every smaller wheel, we proved that $G$
contains a wheel of order at least $2^{h-2}.$
\end{proof}

 \begin{theorem}
Let $k>0$ be an integer and $G$ be a graph. If $\tw(G) \geq 36k - 2$, then $G$ contains a~$\wheel{k}$-model.
 \end{theorem}

 \begin{proof}
  Let $k>0$ be an integer, $G$ be a graph such that $\tw(G) \geq 36k - 2$, and let $h = \ceil{\log 4k}.$
  Since every wheel contains a model of every smaller wheel, we have
  \begin{align*}
    \wheel{k} & \lminor \wheel{2^{\ceil{\log k}} + 1}\\
    & \lminor \wheel{2^{\ceil{(\log 4k) - 2}} + 1}\\
    & \lminor \wheel{2^{h - 2} + 1}\\
  \end{align*}
  Therefore, if we prove that $G$ contains a $\wheel{2^{h - 2} +
    1}$-model, then we are done because the minor relation
  is~transitive. Let $Y_h^-$ be the graph of the 
   following form: the disjoint union of the complete binary tree
   $B_h$ of height $h$ with leaves 
   set $Y_L$ and of the path $Y_P$ on $2^h$ vertices, and let $\mathcal{Y}_h$ be
   the set of graphs of the same form, but with the extra edges $\acc{\acc{l
       ,\phi(l)}}_{l \in Y_L}$, where $\phi : Y_L \to \vertices{Y_P}$
   is a bijection. As we proved in Lemma~\ref{l:wtp} that every graph in $\mathcal{Y}_h$
   contains the wheel of order $2^{h - 2} + 1$ as minor, showing that
   $G$ contains an element of $\mathcal{Y}_h$ as minor suffices to prove this lemma. That is
   what we will do.

   From our initial assumption, we deduce the following.
  \begin{align*}
    \tw(G) & \geq 36k - \frac{5}{2}\\
    & \geq \frac{3}{2} (3 \cdot 2^{\log 8k} - 1) - 1\\
    & \geq \frac{3}{2} (3 \cdot 2^{\floor{\log 4k} + 1} - 1) - 1\\
    \tw(G) & \geq \frac{3}{2} (3 \cdot 2^h - 1) - 1
  \end{align*}

  According to Proposition~\ref{p:leaf},~$G$ has a separation~$(A,B)$ of order~$3 \cdot 2^{h} - 1$ such that
  \begin{enumerate}[(i)]
  \item $\induced{G}{B \setminus A}$ is connected;
  \item $A \cap B$ is linked in~$\induced{G}{B}$; \label{e:item2}
  \item $(A,B)$ left-contains the graph $Y_h^-.$
  \end{enumerate}
 
  By definition of \emph{left-contains}, $\induced{G}{A}$ contains a
  model $(\mathcal{M}^-, \varphi^-)$ of $Y_h^{-}$ and every
  element of $\mathcal{M}^-$ contains exactly one element of $A \cap
  B$. For every $x \in A \cap B$, we denote by $M^-_x$ the element of
  $\mathcal{M}^-$ that contains $x$. 
  Let $L$ (resp.\ $R$) be the subset of $A \cap B$ of vertices that
  belong to an element of $M$ related to the leaves of $B_h$ in $Y_h^{-}$
  (resp.\ to the path $P$). We remark that these sets are both of
  cardinality $2^h$.

  Since $A \cap B$ is linked in~$\induced{G}{B}$ (see~(\ref{e:item2})),
  there is a set $\mathcal{P}$ of $2^h$ disjoint paths between the
  vertices of $L$ and the elements of $R$.
  Let $\psi \colon L \to \vertices{P}$ be the function that match each
  element $l$ of $L$ with the (unique) element of $R$ it is linked to by a
  path (that we call $\mathcal{P}_l$) of $\mathcal{P}$. Observe
  that $\psi$ is a bijection.
  We set
  \begin{align*}
    \forall l \in L,\ M_l &= M^-_l \cup \vertices{\pathtend{l}{\mathcal{P}_l}{\psi(l)}}\\
    \forall r \in (A \cap B) \setminus L,\ M_r &= M^-_r\\
    \\
    \mathcal{M} = \bigcup_{x \in A \cup B} M_x.
  \end{align*}
 Let us show that $\mathcal{M}$ allows us to define a model of some $H \in \mathcal{Y}_h$. Let us consider the following mapping.
  \[
  \varphi \colon
  \left \{
    \begin{array}{ccc}
      \vertices{Y_h^-} &\to& \mathcal{M}\\
      x & \mapsto & M_x
    \end{array}
  \right .
  \]

  We claim that $(\mathcal{M}, \varphi)$ is a model of $H$ for some $H
  \in \mathcal{Y}_h$. This is a consequence of the following remarks.

\begin{remark}
 Every element of $\mathcal{M}$ is either an element of $\mathcal{M}^-$, or the
 union of a element $M$ of $\mathcal{M}^-$ and of the vertices of a path that
 start in $M$, thus every element of $\mathcal{M}$ induces a connected
 subgraph of $G$.
\end{remark}

\begin{remark}
The paths of $\mathcal{P}$ are all disjoint and are disjoint from the
elements of $\mathcal{M}^-$. Every interior of path of $\mathcal{P}$ is in but one
element of $\mathcal{M}$, therefore the elements of $\mathcal{M}$ are
disjoint.
\end{remark}

\begin{remark}
The elements $\acc{m_l}_{l \in L}$ are in bijection with the elements
of $\acc{m_r}_{r \in R}$ (thanks to the function $\psi$) and every two vertices $l \in L$ and $\psi(l) \in
R$ are such that there is an edge between $m_l$ and $m_{\psi(l)}$ (by
definition of $\mathcal{M}^+$).
\end{remark}

We just proved that $(\mathcal{M}, \varphi)$ is a model of a graph
of $\mathcal{Y}_h$ in $G$. Finally, we apply Lemma~\ref{l:wtp} to find a model of the wheel of order $2^{h-2} + 1
= 2^{\ceil{\log k}-2} + 1 \geq k$ in $G$ and this concludes the proof.
 \end{proof}

 \section{Excluding a double wheel with a \texorpdfstring{$(l \log
     l)^2$}{(l * log l)2} bound on treewidth}
 \label{sec:excl_dwheel}

 \begin{definition}[double wheel]
   Let $r>2$ be an integer. The \emph{double wheel} of order $r$
   (denoted $\dwheel{r}$) is a cycle of length $r$ whose each vertex
   is adjacent to two different extra vertices, in other words it is
   the graph of the form
   
   \begin{align*}
     \vertices{G} = &\{o_1, o_2, w_1,\dots, w_r\}\\
     \edges{G} = &\{\acc{w_1,w_2}, \acc{w_2,w_3}, \dots, \acc{w_{r-1}, w_r},
     \acc{w_r, w_1}\}\\ &\cup \acc{\acc{o_1,w_1}, \dots,
       \acc{o_1,w_r}}\\ &\cup \acc{\acc{o_2,w_1}, \dots,
     \acc{o_2,w_r}}
   \end{align*}
   (see Figure~\ref{fig:w6} for an example).
 \end{definition}


 \begin{lemma}\label{l:dwheel}
   Let $G$ be a graph and $h>0$ be an integer. If $tw(G) \geq 6\cdot
   2^h - 4$, then $G$ contains as minor a double wheel of order at
   least~$\frac{2^{\frac{h}{2}}-2}{2h - 3}.$
 \end{lemma} 

 \begin{proof}
   Let $h$ and $G$ be as above.
   Observe that $\tw(G) \geq 3  (2^{h+1} - 1) -1$. As the binary tree
   $T = B_{h}$ has $2^{h+1} - 1$ vertices, $G$ contains a
   graph ${H \in \llam{B_{h}}}$ as minor (by
   Lemma~\ref{l:zeta}). Let us show that any graph $H \in
   \llam{B_{h}}$ contains a double wheel of order at least
   $\frac{2^{\frac{h}{2}}-2}{2h - 3}$ as minor.

   Let $P$ be the path of length at least $2^{\frac{h}{2}}$ in the 
   definition of  $H$. Let $L$ be the set, of size at least $2^{\frac{h}{2}}$, of the
   leaves of $T$ that are adjacent to $P$ in $H$. Such a set exists by
   definition of $\llam{B_{h}}.$
   We also define $u$ (resp.\ $u'$) as the vertex of $\leaves{T}$ that
   is adjacent to one end of $P$ (resp.\ to the other end of $P$) and~${Q =
     \patht{u}{T}{u'}}$.

   As $T$ is a binary tree of height $h$, $Q$ has at
   most $2h - 1$ vertices. Each vertex of $Q$ is of degree at
   most~3 in $T$ except the two ends which are of degree~1. Consequently, $T \setminus Q$ has at most $2h - 3$
   connected components that are subtrees of $T.$
   Notice that every vertex of the $2^{\frac{h}{2}}$ elements of $L$
   is either a leaf of one of these $2h -3$ subtrees, or one of the
   two ends of $Q$.
   By the pigeonhole principle, one of these subtrees, which we call $T_1$, has at least
   $\frac{2^{\frac{h}{2}} - 2}{2h - 3}$ leaves that are elements of~$L.$

   Let $M_{o_1}$ be the set of vertices of this subtree $T_1$. We also
   set $M_{o_2} =  \acc{v_{\rm new}}$ (cf.~Definition~\ref{d:llam} for a
   definition of $v_{\rm new}$). Let us consider the cycle $C$ made by
   the concatenation of the paths $\patht{u}{P}{u'}$
   and~$\patht{u'}{T}{u}$ in $H$.

   By definition of $M_{o_1}$, there are at least
   $\frac{2^{\frac{h}{2}} - 2}{2h - 3}$ vertices of $C$ adjacent to
   vertices of~$M_{o_1}$. Let $J = \acc{j_1, \dots, j_{\card{J}}}$ be the
   set of such vertices of~$C$, in the same order as they appear in
   $C$ (we then have $\card{J} \geq \frac{2^{\frac{h}{2}} - 2}{2h - 3}$).
   
   We arbitrarily choose an orientation of $C$ and define the sets of vertices $M_1, M_2, \dots, M_{\card{J}}$ as follows.
   \begin{align*}
     \forall i \in \intv{1}{\card{J} - 1},\ M_i & =
     \vertices{\pathtend{j_i}{C}{j_{i + 1}}}\\
     M_{\card{J}} & = \vertices{\pathtend{j_{\card{J}}}{C}{j_{1}}}
   \end{align*}

   Let $\mathcal{M} = \acc{M_1, \dots, M_{\card{J}}, M_{o_1},
     M_{o_2}}$ and $\psi \colon \vertices{\dwheel{\card{J}}} \to
   \mathcal{M}$ be the function defined by
   \begin{align*}
     \forall i \in \intv{1}{\card{J}},\ \psi(w_i) & = M_i\\
     \psi(o_1) & = M_{o_1}\\
     \psi(o_2) & = M_{o_2}
   \end{align*}
   Notice that $\psi$ maps the vertices of $\dwheel{\card{J}}$ to connected subgraphs of $H$
   such that $\forall (v,w) \in \edges{\dwheel{\card{J}}}$, there is a
   vertex of $\psi(v)$ adjacent in $H$ to a vertex of
   $\psi(w)$. Therefore, $(\mathcal{M}, \psi)$ is a
   $\dwheel{\card{J}}$-model in $H.$

   Since $\card{J} \geq \frac{2^{\frac{h}{2}} - 2}{2h - 3}$, $H$ contains
   a double wheel of order at least $\frac{2^{\frac{h}{2}} - 2}{2h - 3},$
   which is what we wanted to show.
 \end{proof}

 \begin{corollary}\label{c:dwheel}
   Let $l>0$ be an integer and $G$ be a graph. If $\tw(G) \geq 12 l -
   4$ then $G$ contains a double wheel of order at least
   $\frac{\sqrt{l} - 2}{2 \log l -5}$ as minor.
 \end{corollary}

 \begin{proof}
   Let $l$ and $G$ be as above. First remark that
   \begin{align}
\ceil{\log l} - 1 & \leq \log l \leq \ceil{\log l} \label{eqn1}
   \end{align}
   Our initial assumption on $\tw(G)$ gives the following.
   \begin{align*}
     \tw(G) & \geq 12 l - 4\\
     & \geq 6 \cdot 2^{\log(2l)} - 4\\
     & \geq 6 \cdot 2^{\log l + 1} - 4\\
     & \geq 6 \cdot 2^{\ceil{\log l}} - 4&\quad \text{by}\ (\ref{eqn1})
   \end{align*}
By Lemma~\ref{l:dwheel}, $G$ contains a double wheel of order at
least 
\begin{align*}
  q & = \frac{2^{\frac{\ceil{\log l}}{2}} - 2}{2\ceil{\log l} - 3}\\
  & \geq \frac{2^{\half \log l} - 2}{2 (\log l - 1) - 3}&\quad \text{by}\ (\ref{eqn1})\\
  & \geq \frac{\sqrt{l} - 2}{2 \log l - 5}
\end{align*}
Therefore, $G$ contains a double wheel of order $q \geq \frac{\sqrt{l}
  - 2}{2 \log l - 5}$, as required.
 \end{proof}

 \begin{theorem}[follows from Corollary~\ref{c:dwheel}]
   Let $k>0$ be an integer and $G$ be a graph. If
   $\tw(G) \geq 12 (8k \log(8k) + 2)^2 -4$, then $G$ contains a double
   wheel of order at least $k$ as minor.
 \end{theorem}

 \begin{proof}
   Applying Corollary~\ref{c:dwheel} for $l = (8k \log(8k) + 2)^2$
   yields that $G$ contains a double wheel of order at least
   \begin{align*}
     q & \geq \frac{\sqrt{l} - 2}{2 \log l - 5}\\
     & \geq \frac{8k \log(8k)}{4 \log(8k \log(8k) + 2) - 5}\\
     & \geq \frac{8k \log(8k)}{4 \log(8k \log(8k)) - 1}\\
     & \geq \frac{8k \log(8k)}{4 (\log(8k) + \log \log(8k)) - 1}\\
     & \geq \frac{8k \log(8k)}{8 \log(8k) - 1}\\
     & \geq k
   \end{align*}
Consequently $G$ contains a double wheel of order $q \geq k$ and we are done.
 \end{proof}

 \section{Excluding a graph of pathwidth at most 2 with a quadratic bound on treewidth}
 \label{sec:excl_sq}

\begin{definition}[graph $\sg{r}$]\label{d:sg}
 We define the graph $\Xi_{r}$ as the graph of the following form (see figure~\ref{fig:sg}).
\[
\left \{
  \begin{array}{l}
    \vertices{G} = \{x_0, \dots, x_{r-1},y_0, \dots, y_{r-1},z_0,
    \dots, z_{r-1}\}\\
    \edges{G} = \{\acc{x_i,x_{i+1}}, \acc{z_i,z_{i+1}}\}_{i \in \intv{1}{r-1}}
    \cup \{\acc{x_i,y_{i}}, \acc{y_i,z_i}\}_{i \in \intv{0}{r-1}}
  \end{array}
\right .
\]
\end{definition}

\begin{figure}[ht]
  \centering
  {\small 
  \begin{tikzpicture}
    \foreach \x in {0,...,4}
    {
      \draw (\x,0) node {};
      \node[normal] at (\x + .25, .25) {$z_{\x}$};
    }
    \foreach \x in {0,...,4}
    {
      \draw (\x,1) node {};
      \node[normal] at (\x + .25, 1.25) {$y_{\x}$};
    }
    \foreach \x in {0,...,4}
    {
      \draw (\x,2) node {};
      \node[normal] at (\x + .25, 2.25) {$x_{\x}$};
    }
    \foreach \x in {0,...,4}
    {
      \draw (\x,0) -- (\x,1);
      \draw (\x,1) -- (\x,2);
      \draw (\x,0) -- (+1,0);
      \draw (\x,2) -- (+1,2);
    }
  \end{tikzpicture}
  }
\caption{The graph $\sg{5}$}
  \label{fig:sg}
\end{figure}
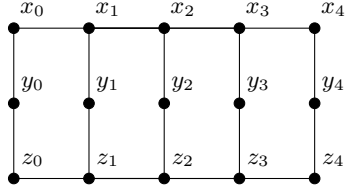

\subsection{Graphs of pathwidth 2 in
  \texorpdfstring{$\sg{r}$}{the vertically subdivided 2xr grid}}

Instead of proving that a treewidth quadratic in $|\vertices{H}|$
forces an $H$-minor for every graph $H$ of pathwidth at most 2, we
prove that a treewidth quadratic in $r$ forces an $\sg{r}$-minor and
then that every graph of pathwidth at most 2 on~$r$ vertices is minor
of~$\sg{r}.$ 
For this, we first need some lemmata and remarks about path decompositions.

\begin{definition}[nice path decompostion, \cite{Klo94}]
 A path decomposition $\paren{p_1 p_2 \dots p_k, \{X_{p_i}\}_{i \in
   \intv{1}{k}}}$ of a graph $G$ is said to be \emph{nice} if
 $\card{X_{p_1}} = 1$ and
\[
\forall i \in \intv{2}{k},\ \card{(X_{p_i} \setminus X_{p_{i-1}}) \cup
(X_{p_{i-1}} \setminus X_{p_{i}})} = 1
\]
It is known \cite{BodlaenderT04-Co} that every graph have an optimal
path decomposition which is nice and that in such decomposition, every
node $X_i$ is either an \emph{introduce node} (\ie~either $i = 1$ or $\card{X_{p_i}
  \setminus X_{p_{i-1}}} = 1$) or a \emph{forget node} (\ie~$\card{X_{p_{i-1}}
  \setminus X_{p_i}} = 1$).
\end{definition}

\begin{remark}
  It is easy to observe that for every graph $G$ on $n$ vertices, there is an optimal
  path decomposition with $n$ introduce nodes and $n$ forget nodes
  (one of each for each vertex of $G$), thus of length $2n$.
\end{remark}

\begin{remark} \label{r2}
  Let $G$ be a graph and let $\paren{p_1 p_2 \dots p_k, \mathcal{X}}$,
    $\mathcal{X} = \acc{X_{p_i}}_{i \in \intv{1}{k}}$ be a nice (non
    necessarly optimal) path decomposition of~$G$. Let $w$ be the
    width of this decomposition.

    For every $i \in \intv{2}{k-1}$, if $p_i$ is a forget node,
    $\card{X_{p_i}}\leq w - 1$ and $p_{i+1}$ is an introduce node,
    then by setting
    \begin{align*}
      X'_{p_i} & = X_{p_{i-1}} \cup X_{p_{i+1}}\\
      \forall j \in \intv{1}{k},\ j \neq i,\ X'_{p_j} & = X_{p_j}\\
      \mathcal{X}' & = \acc{X'_{p_j}}_{j \in \intv{1}{k}} 
    \end{align*}
    we create from $\paren{p_1 p_2 \dots p_k, \mathcal{X}'}$ a valid path
    decomposition of $G$, where $p_i$ is now an introduce node and
    $p_{i+1}$ a forget node. Observe that $\card{X'_{p_i}} \leq
    \card{X_{p_i}} + 2 = w+1$ Therefore the new path decomposition
    has the same width as the original~one. Note that the condition
    $\card{X_{p_i}}\leq w - 1$ holds, for instance, when $p_{i-1}$ is
    required to be a forget node too (for $i \in \intv{3}{k-1}$).
\end{remark}

\begin{remark} \label{r3}
  Let $G$ be a graph and $P = \paren{p_1 p_2 \dots p_k, \mathcal{X}}$
  be a nice path decomposition of $G.$
  For every $i \in \intv{1}{k}$, the path $p_1\dots p_i$ contains at
  most as many forget nodes as introduce nodes and the difference
  between these two numbers is at most $w + 1$ where $w$ is the width
  of $P.$
\end{remark}

 \begin{lemma}\label{l:compact_pathdec}
   Let $G$ be a graph on $n$ vertices
   . Then
   $G$ has an optimal path decomposition $P$ such that
   \begin{enumerate}[(i)]
   \item every bag of $P$ has size $\pw(G) + 1$;
   \item every two adjacent bags differs by exactly one element,
     \ie~for every two adjacent vertices $u$ and $v$ of $P$,
     $\card{X_u \setminus X_v} = \card{X_v \setminus X_u} = 1$.
   \end{enumerate}
 \end{lemma}

 \begin{proof}
   Let $P = \paren{p_1 p_2 \dots p_{2k}, \mathcal{X}}$ with $\mathcal{X} = \{X_{p_i}\}_{i \in \intv{1}{2k}}$ be
   a nice optimal path decomposition of $G$ with as many introduce nodes
   (resp. forget nodes) as there are vertices in $G.$

   Let $s = \pw(G) + 1$.
   According to Remarks~\ref{r2} and~\ref{r3}, $P$ can be modified
   into a path decomposition of $G$ of the same width and such that
   \begin{enumerate}[(a)]
   \item the $s$ first vertices of $P$ are introduce nodes and $p_{s +
       1}$ is a forget node;
   \item the $s$ last vertices of $P$ are forget nodes and $p_{2k -
       s}$ is an introduce node;
   \item for every $i \in \intv{s}{2k - s}$, $p_i$ and $p_{i+1}$
     are nodes of different type. 
   \end{enumerate}
   In the arguments to follow, we assume that $P$ satisfies this property.

   \begin{remark}
     Introduce nodes all have bags of cardinality $s$. \label{r:card}
   \end{remark}

   \begin{remark}
   For every $i \in \intv{0}{k - s}$, the node $p_{s + 2i}$ is an
   introduce node and the node $p_{s + 2i + 1}$ is a forget node, which
   implies $X_{p_{s + 2i}} \subsetneq X_{p_{s + 2i +1}}$. Also note
   that for every $i \in \intv{1}{s - 1},\ X_{p_i} \subsetneq X_{p_s}$ and
   for every $i \in \intv{2k - s + 1}{2k}$, $X_{p_i} \subsetneq X_{p_{2k -
     s}}$.     
   \end{remark}

   Intuitively, every bag $X$ that is included in one of its adjacent
   bags $X'$ contains no more information than what $X'$ already contains,
   so we will just remove it.
   
   We thus define $P' = p_{s} p_{s + 2} \dots p_{s + 2i} \dots p_{2k-s}$
   (a path made of all introduce nodes of $P$). Clearly, $P$ and $P'$
   have the same width and as we deleted only redundant nodes, $P'$ is
   still a valid path decomposition of $G$.

   Since every two adjacent nodes of $P'$ were introduce nodes separated
   by a forget node in $P$, they only differ by one element. According
   to Remark~\ref{r:card} and since every node of $P'$ was an
   introduce node in $P$, every bag of $P'$ have size $\pw(G) + 1.$
   Consequently, $P'$ is an optimal path decomposition that satisfies
   the conditions of the lemma statement.
 \end{proof}

 \begin{remark}\label{c:compact}
The path decomposition of Lemma~\ref{l:compact_pathdec} has length $\vertices{G} - \pw(G).$
 \end{remark}

 \begin{proof}
   Let $(P, \mathcal{X})$ be such a path decomposition.
   Remember that the first node of $P$ has a bag of size $\pw(G) + 1$ and that
   every two adjacent nodes of $P$ have bags which differs by exactly one
   element. Since every vertex of $G$ is in a bag of $P$, in addition
   to the first bag containing $\pw(G) + 1$ vertices of $G$, $P$
   must have $\vertices{G} - \pw(G) -1$ other bags in order to contain
   all vertices of $G$. Therefore $P$ has length $\vertices{G} - \pw(G).$
 \end{proof}

A proof of a slightly weaker version of the following lemma previously
appeared~\cite{Pro89}.
 \begin{lemma}\label{l:pro}
   For every graph~$G$ on~$n$ vertices and of pathwidth at most 2,
   there is a minor model of~$G$ in~$\sg{n-1}.$
 \end{lemma}

 \begin{proof}
   Let $G$ be as in the statement of the lemma. We assume that $\pw(G)
   = 2$ (if this is not the case we add edges to $G$ in order to
   obtain a graph of pathwidth 2 which contains $G$ as a minor). Let $r =
   \vertices{G} - \pw(G) = n-2$.
   
   Let $P = (p_1\dots p_{r}, \acc{X_{p_1}, \dots, X_{p_{r}}})$ be an optimal path decomposition of $G$
   satisfying the properties of Lemma~\ref{l:compact_pathdec}, of
   length $r$. Such decomposition exists according to
   Lemma~\ref{l:compact_pathdec} and Remark~\ref{c:compact}).

   Using this decomposition, we will now define a labeling~$\lambda$ of the
   vertices of~$\sg{r+1}$. When dealing with the vertices of~$\sg{r+1}$ we
   will use the notations defined in Definition~\ref{d:sg}.
   Let~$\lambda \colon \vertices{\sg{r+1}} \to \vertices{G}$ be the
   function defined as follows:
   \begin{enumerate}[(a)]
   \item $\lambda(x_0)$ and $\lambda(y_0)$ are both equal to one
     (arbitrarily chosen) element of the set $X_{p_1} \cap X_{p_2};$
   \item $\lambda(z_0)$ is equal to the only element of the set $X_{p_1} \cap X_{p_2}
     \setminus \{\lambda(x_1)\};$
   \item $\forall i \in \intv{2}{r}$, $\lambda(y_i) = X_{p_i}
     \setminus X_{p_{i-1}}$ and we consider two cases:
     \begin{enumerate}[{Case} 1:]
     \item \label{e:case1}$X_{p_{i-1}} \cap X_{p_{i}} =
       X_{p_{i}} \cap X_{p_{i + 1}}$

       $\lambda(x_i) = \lambda(x_{i-1})$ and $\lambda(z_i) = \lambda(z_{i-1});$

     \item \label{e:case2}$X_{p_{i-1}} \cap X_{p_{i}} \neq
       X_{p_{i}} \cap X_{p_{i + 1}}$

if $X_{p_{i-1}} \cap X_{p_i} \cap X_{p_{i+1}} = \lambda(x_{i-1}),$
       \begin{itemize}
       \item[then] $\lambda(x_i) = \lambda(x_{i-1})$ and $\lambda(z_i)
         = X_{p_i} \setminus X_{p_{i-1}}$;
       \item[else] $\lambda(x_i) = X_{p_i} \setminus X_{p_{i-1}}$ and
         $\lambda(z_i) = \lambda(z_{i-1}).$
       \end{itemize}
     \end{enumerate}
   \end{enumerate}

   Thanks to this labeling, we are now able to present a minor model
   of~$G$ in~$\sg{r+1}:$

   \begin{align*}
     \forall v \in \vertices{G},\ M_v & = \acc{u \in \vertices{\sg{r+1}},\
       \lambda(u) = v}\\
     \mathcal{M} &=  \acc{M_v}_{v \in
       \vertices{G}}\\\\
     \varphi &\colon \left \{
     \begin{array}{lll}
       \vertices{G} &\to &\mathcal{M}\\
       u&\mapsto &M_u
     \end{array}\right.
     \end{align*}

To show that $(\mathcal{M}, \varphi)$ is a $G$-model in~$\sg{r+1}$, we now
check if it matches the definition of a minor model.

By definition, every element of~$\mathcal{M}$ is a subset
of~$\vertices{\sg{r+1}}$. To show that every element of~$\mathcal{M}$ 
induces a connected subgraph in~$G$, it suffices to show that nodes
of~$\sg{r+1}$ which have the same label induces a connected subgraph
in~$G$ (by construction of the elements of $\mathcal{M}$). This can
easily be seen by remarking that for every $i \in \intv{2}{r}$, every
vertex $y_i$ of $\sg{r+1}$ gets a new label and that every vertex $x_i$
(resp.\ $z_i$) of $\sg{r+1}$ receive either the same label as $y_i$, or
the same label as $x_{i-1}$ (resp.\ $z_{i-1}$).

Let us show that this labeling ensure that if two vertices $u$ and $v$ of
$G$ are in the same bag of $P$, there are two adjacent vertices of $\sg{r+1}$
that respectively gets labels $u$ and $v.$
Let $u,v$ be two vertices of $G$ which are in the same bag of $P$. Let
$i$ be such that $X_i$ is the first bag of $P$ (with respect to the
subscripts of the bags of $P$) which contains both $u$ and $v.$
The case $i = 1$ is trivial so we assume that $i>1$. We also assume
without loss of generality that $X_i \setminus X_{i - 1} = \{v\},$
what gives $\lambda(y_i) = v$. Depending on in what case we are,
either either $\lambda(x_i) = u$~(\ref{e:case1}) or $\lambda(z_i)
= u$ (\ref{e:case1} and~\ref{e:case2}). In both cases, $u$
and $v$ are the labels of two adjacent nodes of $\sg{r+1}.$
By construction of the elements of $\mathcal{M}$, this implies that if
$\{u,v\} \in \edges{G}$, then there are vertices $u' \in \varphi(u)$
and $v' \in \varphi(v)$ such that $\{u',v'\} \in \edges{\sg{r+1}}.$

Therefore, $(\mathcal{M}, \varphi)$ is a $G$-model in~$\sg{n-1}$, what
we wanted to find.

 \end{proof}

\subsection{Exclusion of \texorpdfstring{$\sg{r}$}{the vertically
    subdivided 2xr grid}}

\begin{lemma}\label{l:path_linked}
For any graph, if $\tw(G) \geq 3 \ell - 1$ then $G$ contains as minor
the following graph: a path $P = p_1\dots p_{2\ell}$ of length~$2
\ell$ and a family $Q$ of $\ell$ paths of length 2 such that every
vertex of $P$ is the end of exactly one path of $Q$ and every path of
$Q$ has one end in $p_1\dots p_{l}$ (the first half of $P$) and the
other end in $p_{l+1}\dots p_{2l}$ (the second half of $P$) (see figure
\ref{fig:path_linked}).
\end{lemma}

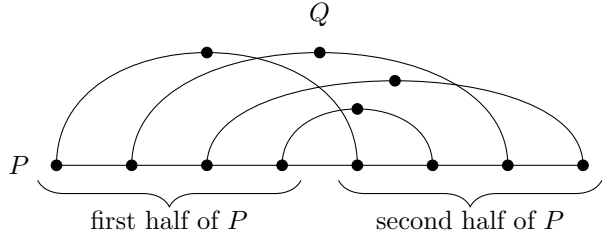
\begin{figure}[ht]
  \centering
  \begin{tikzpicture}
    \draw (0,0) node {}
    \foreach \x in {1,...,7}
    {
      -- ++(1,0) node {}
    };
    \node[normal] at (-.5,0) {$P$};
    \node[normal] at (3.5,2) {$Q$};
    \draw (0,0) .. controls +(0,2) and +(0,2) .. (4,0) node[midway] {};
    \draw (1,0) .. controls +(0,2) and +(0,2) .. (6,0) node[midway] {};
    \draw (2,0) .. controls +(0,1.5) and +(0,1.5) .. (7,0) node[midway] {};
    \draw (5,0) .. controls +(0,1) and +(0,1) .. (3,0) node[midway] {};
    \draw [decorate,decoration={brace,amplitude=10pt}]
    (3.25,-.2) -- (-0.25,-.2) node [normal, black,midway,yshift=-15pt] {first half of $P$};
    \draw [decorate,decoration={brace,amplitude=10pt}]
    (7.25,-.2) -- (3.75,-.2) node [normal, black,midway,yshift=-15pt] {second half of $P$};
  \end{tikzpicture}
\caption{Example for Lemma~\ref{l:path_linked}}
  \label{fig:path_linked}
\end{figure}

\begin{proof}
  Let $\ell> 0$ be an integer and $G$ be a graph of treewidth at least
  $3 \ell - 1$. According to Proposition~\ref{p:leaf}, $G$ has a
  separation~$(A,B)$ of order~$2\ell$ such that
  \begin{enumerate}[(i)]
  \item $\induced{G}{B \setminus A}$ is connected;
  \item $A \cap B$ is linked in~$\induced{G}{B}$;
  \item $(A,B)$ left-contains a path $P = p_1\dots p_{2 \ell}$ of length $2\ell.$
  \end{enumerate}

Let $(\mathcal{M}, \varphi)$ be a model of $P$ in $\induced{G}{A},$
with $\mathcal{M} = \{M_1, \dots, M_{2 \ell}\}$. We assume without loss
of generality that $\varphi$ maps~$p_i$ to~$M_i$ for every $i \in
\intv{1}{2 \ell}$.

As $A \cap B$ is linked in~$\induced{G}{B}$, there is a set~$Q$
of~$\ell$ disjoint paths in~$\induced{G}{B}$ of length at least 2 and
such that 
every path~$q \in Q$ has one end in $(A \cap B) \cap \bigcup_{i \in
  \intv{1}{\ell}} M_i$, the other end in $(A \cap B) \cap
\bigcup_{i \in \intv{\ell +1}{2\ell}} M_i$ and its internal vertices
are not in~$A \cap B.$

Let $G'$ be the graph obtained from $\induced{G}{\paren{\bigcup_{q \in Q}\vertices{q}}
  \cup \paren{\bigcup_{M \in \mathcal{M}} M}}$ after the following operations.
\begin{enumerate}
\item iteratively contract the edges of every path of $Q$ until it
  reaches a length of 2. The paths of $Q$ have length at least 2, so this
  is always possible.
\item for every $i \in \intv{1}{2 \ell}$, contract $M_i$ to a single
  vertex. The elements of a model are connected (by definition) thus
  this operation can always be performed.
\end{enumerate}
As one can easily check, the graph $G'$ is the graph we were
looking for and it has been obtained by contracting some edges of a
subgraph of $G$, therefore $G' \lminor G$.
\end{proof}

\begin{theorem}
  Let $G$ be a graph and $H$ be a graph on $h$
  vertices satisfying $\pw(H) \leq 2$. If $\tw(G) \geq 3(h-2)² -1$ then $G$
  contains $H$ as a minor.
\end{theorem}

\begin{proof}
  Let $G$, $H$ and $h$ be as in the statement of the Lemma.
  According to Lemma~\ref{l:pro}, every graph of pathwidth at most two
  on $n$ vertices is minor of~$\sg{n-1}$. Therefore in order to show
  that $G \lminor H$ it is enough to prove that $G \lminor
  \sg{h-1}$. This is what we will do.

  According to Lemma~\ref{l:path_linked}, $G$ contains as minor two paths $P =
 p_1\dots p_{(h-2)²}$ and $R = r_1 \dots r_{h-2)²}$ and a family
$Q$ of $(k-2)²$ paths of length 2 such that every 
vertex of $P$ or $R$ is the end of exactly one path of $Q$ and every path of
$Q$ has one end in $P$ and the other end in $R$. For every $p \in P$,
we denote by $\varphi(p)$ the (unique) vertex of $R$ to which $p$ is linked to by a
path of $Q$. Observe that $\varphi$ is a bijection.
By Proposition~\ref{p:es}, there is a subsequence $P'=(p'_1,p'_2, \dots,
p'_{h-1})$ of the vertices of $P$ such that the vertices
$\varphi(p'_1),\varphi(p'_2), \dots, \varphi(p'_{h-1})$ appear in $R$
either in this order or in the reverse order. Let $R'
=(\varphi(p'_1),\varphi(p'_2), \dots, \varphi(p'_{h-1}))$ and $Q'$ be
the set of inner vertices of the paths from $p'_i$ to $\varphi(p'_i)$
for all $i \in \intv{1}{h-1}.$

Iteratively contracting in $G$
 which have at most one end in $P'$ (resp.\ in $R'$) and removing the
vertices that are not in $P'$, $R'$ or $Q'$ gives the graph $\sg{h-1}$. The
operations used to obtain it are vertices and edge deletions, and edge
contractions, thus $\sg{h-1}$ is a minor of $G$. This concludes the~proof.
\end{proof}

\section{Excluding a yurt graph}
\label{sec:excl_yurt}

\begin{definition}[yurt graph of order~$r$]
  Let~$r>0$ be an integer. In this paper, we call \emph{yurt graph} of order~$r$ the graph~$Y_r$ of the form
  \begin{align*}
    \vertices{Y_r}  =& \acc{x_1, \dots, x_r, y_1, \dots, y_r, o}\\
    \edges{Y_r}  =& \acc{\{x_i, x_{i+1}\}_{i \in \intv{1}{r-1}}}\\
                     &\cup \acc{\{y_i, y_{i+1}\}_{i \in
                       \intv{1}{r-1}}}\\
    &\cup \acc{\{x_i,y_i\}}_{i \in \intv{1}{r}}\\
    & \cup \acc{\{y_i,o\}}_{i \in \intv{1}{r}}
  \end{align*}
(see Figure~\ref{fig:y} for an example).
\end{definition}

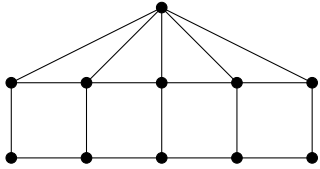
\begin{figure}[ht]
  \centering
  \begin{tikzpicture}
    \foreach \x in {0,...,4}
    {
      \draw (\x,0) node {};
    }
    \foreach \x in {0,...,4}
    {
      \draw (\x,1) node {};
      \draw(\x,1) -- (2, 2);
      \draw (\x,1) -- (\x,0);
    }
      \draw (2, 2) {};
      \draw (0,0) -- (4,0);
      \draw (0,1) -- (4,1);
      \draw (2,2) node {};
  \end{tikzpicture}
\caption{The yurt graph of order 5, $Y_{5}$}
  \label{fig:y}
\end{figure}

For every $r>0$, we define the \emph{comb of order $r$} as the tree
made from the path $p_1p_2\dots p_r$ and the extra vertices $v_1,
v_2, \dots, v_r$ by adding an edge between $p_i$ and $v_i$ for every
$i \in \intv{1}{r}.$

  \begin{theorem}
    Let $k>0$ be an integer and $G$ be a graph. If $\tw(G) \geq 6 k^4
    - 24 k^3 + 48 k^2 - 48 k + 23$, then $G$ contains $Y_k$ as minor.
  \end{theorem}

  \begin{proof}
    Let $k>0$ be an integer and $G$ be a graph such that $\tw(G) \geq
    6 k^4 - 24 k^3 + 48 k^2 - 48 k + 23.$
    Let $C$ be the comb with $l = k^4 - 4 k^3 + 8 k^2 - 8 k + 4$
    teeth.
    As $\tw(G) \geq 3 \card{\vertices{C}} - 1$, $G$ contains
    some graph of $\llam{C}$ by Lemma~\ref{l:zeta}.

    Let us prove that every graph of $\llam{C}$ contains the yurt
    graph of order $k$. Let $H$ be a graph of $\llam{C}.$
    We respectively call $T$, $P$ and $o$ the tree, path and extra
    vertex of $\llam{C}$. Let $F$ be the subset of edges between $P$
    and the leaves of~$T$
    
    Let $L = l_0, \dots, l_{k^2 - 2k + 2}$ (resp.\ $Q = q_0, \dots, q_{k^2 -
      2k + 2}$) be the leaves of $T$ (resp.\ of $P$)that are
    the end of an edge of $F$ We assume without loss of generality
    that they appears in this order.

    According to Proposition~\ref{p:es}, there is a subsequence $Q'$ of $Q$
    of length $k$ such that the corresponding vertices $L'$ of $L$
    appear in the same order. As one can easily see, this graph
    contains the yurt of order $k$ and we are done.
  \end{proof}

\section{Discussion and open problems}
An natural question is whether the  results of this paper  for the classes ${\cal H}_{k}^{i}, i\in\{1,\ldots,3\}$ are tight. This is indeed the case for the wheels in ${\cal H}_{k}^{1}$
as~\eqref{sgdb} does not hold for any pair of the form $({\cal H}_{k}^{1}, f(k))$ where $f=o(k)$. 
To see this, it is enough to observe that 
a clique $K_{k}$ does not contain any wheel on $k+1$ vertices 
as a minor while has treewidth $k-1=\Omega(k)$.  Clearly, 
the same lower bound holds for  
 ${\cal H}_{k}^{2}$ (i.e., the double wheels).

It is easy to prove that~\eqref{sgdb} does 
not hold for any pair of the form $({\cal H}_{k}^{0}, f(k))$ or $({\cal H}_{k}^{3}, f(k))$ when 
$f=o(k \log k)$.
To see this, consider a  (large enough)  $n$-vertex 3-regular Ramanujan graph $R$ 
(see~\cite{Morgenstern1994exis}).
Such a graph has girth at least $c\log n$ for some universal 
constant $c$ (see~\cite{Biggs1990note}), and satisfies~$\tw(R)=\Omega(n)$
(cf.~\cite[Corollary~1]{Bezrukov2004155}).
Let $k'$ be the minimum integer such that $n< k' \cdot c\log n$~holds. Notice  
that $n=\Omega(k'\log k')$, thus $\tw(R)=\Omega(k'\log k')$.
We will show that no graph of ${\cal H}_{2k'}^{0}\cup {\cal
  H}_{2k'}^{3}$ is a minor of~$R$.
As every graph of ${\cal H}_{2k'}^{0}\cup {\cal H}_{2k'}^{3}$ contains
$k' \cdot K_3$ as a minor, it is enough to show that $k' \cdot K_3$ is not a
minor of $R$.
If $k' \cdot K_3$ is a minor of $R$, then $R$ contains a collection
of $k'$ vertex-disjoint cycles. As the girth of $R$
is at least $c\log n$, we have that $n\geq k' \cdot c\log n$, a
contradiction.\medskip

The above observation implies that the function $f(k)=\Theta(k\log k)$ is the best for which~\eqref{sgdb} may  hold for 
the pairs $({\cal H}_{k}^{0}, f(k))$ and $({\cal H}_{k}^{3}, f(k))$ and we conjecture that this is indeed 
the case. Observe that, by the same remark, the lower bound
$\Omega(k\log k)$ also holds for any class
$\{H_k\}_{k \in \N}$ such that for every $k \in \N$, $H_k$ contains as a minor $\Omega(k)$ vertex disjoint cycles.
Interestingly, the above proof does not apply for the double wheels in
${\cal H}_{k}^{2}$. This tempts us to conjecture that~\eqref{sgdb}  
holds (optimally) for the pair $({\cal H}_{k}^{2}, k)$.

\paragraph{Acknowledgement.} We wish to thank Konstantinos Stavropoulos for bringing the results in~\cite{LeafS12sube} 
(and, in particular, Proposition~\ref{p:leaf}) to our attention, during   Dagstuhl Seminar 11071.

%


\begin{thebibliography}{10}

\bibitem{Bezrukov2004155}
S.~Bezrukov, R.~Elsässer, B.~Monien, R.~Preis, and J.-P. Tillich.
\newblock New spectral lower bounds on the bisection width of graphs.
\newblock {\em Theoretical Computer Science}, 320(2–3):155 -- 174, 2004.

\bibitem{Bie}
D.~Bienstock, N.~Robertson, P.~Seymour, and R.~Thomas.
\newblock Quickly excluding a forest.
\newblock {\em J. Combin. Theory Ser. B}, 52(2):274--283, 1991.

\bibitem{Biggs1990note}
N.~Biggs and A.~Boshier.
\newblock Note on the girth of ramanujan graphs.
\newblock {\em Journal of Combinatorial Theory, Series B}, 49(2):190 -- 194,
  1990.

\bibitem{BirmeleBR07bram}
E.~Birmel{\'e}, J.~Bondy, and B.~Reed.
\newblock Brambles, prisms and grids.
\newblock In A.~Bondy, J.~Fonlupt, J.-L. Fouquet, J.-C. Fournier, and J.~L.
  Ram{\'\i}rez~Alfons{\'\i}n, editors, {\em Graph Theory in Paris}, Trends in
  Mathematics, pages 37--44. Birkh{\"a}user Basel, 2007.

\bibitem{Bodl93a}
H.~L. Bodlaender.
\newblock On linear time minor tests with depth-first search.
\newblock {\em J. Algorithms}, 14(1):1--23, 1993.

\bibitem{Bodlaender98apa}
H.~L. Bodlaender.
\newblock A partial {$k$}-arboretum of graphs with bounded treewidth.
\newblock {\em Theoret. Comput. Sci.}, 209(1-2):1--45, 1998.

\bibitem{BodlaenderT04-Co}
H.~L. Bodlaender and D.~M. Thilikos.
\newblock Computing small search numbers in linear time.
\newblock In {\em Proceedings of the First International Workshop on
  Parameterized and Exact Computation (IWPEC 2004)}, volume 3162 of {\em LNCS},
  pages 37--48. Springer, 2004.

\bibitem{BodlaenderLTT97onin}
H.~L. Bodlaender, J.~van Leeuwen, R.~B. Tan, and D.~M. Thilikos.
\newblock On interval routing schemes and treewidth.
\newblock {\em Inf. Comput.}, 139(1):92--109, 1997.

\bibitem{ChekuriC13poly}
C.~Chekuri and J.~Chuzhoy.
\newblock Polynomial bounds for the grid-minor theorem.
\newblock {\em CoRR}, abs/1305.6577, 2013.

\bibitem{DemaineH08line}
E.~D. Demaine and M.~Hajiaghayi.
\newblock Linearity of grid minors in treewidth with applications through
  bidimensionality.
\newblock {\em Combinatorica}, 28(1):19--36, 2008.

\bibitem{DemaineHK09}
E.~D. Demaine, M.~Hajiaghayi, and K.~Kawarabayashi.
\newblock Algorithmic graph minor theory: Improved grid minor bounds and
  {W}agner's contraction.
\newblock {\em Algorithmica}, 54(2):142--180, 2009.

\bibitem{DiestelJGT99high}
R.~Diestel, T.~R. Jensen, K.~Y. Gorbunov, and C.~Thomassen.
\newblock Highly connected sets and the excluded grid theorem.
\newblock {\em J. Combin. Theory Ser. B}, 75(1):61--73, 1999.

\bibitem{ErdosSzekeres}
P.~Erd\H{o}s and G.~Szekeres.
\newblock A combinatorial problem in geometry.
\newblock In I.~Gessel and G.-C. Rota, editors, {\em Classic Papers in
  Combinatorics}, Modern Birkh\"{a}user Classics, pages 49--56. Birkh\"{a}user
  Boston, 1987.

\bibitem{FellowsL94}
M.~R. Fellows and M.~A. Langston.
\newblock On search, decision, and the efficiency of polynomial-time
  algorithms.
\newblock {\em J. Comput. System Sci.}, 49(3):769--779, 1994.

\bibitem{FominLS12bidi}
F.~V. Fomin, D.~Lokshtanov, and S.~Saurabh.
\newblock Bidimensionality and geometric graphs.
\newblock In {\em 23st ACM--SIAM Symposium on Discrete Algorithms (SODA 2012)}.
  ACM-SIAM, San Francisco, California, 2012.

\bibitem{kawarabayashiK2line}
K.~ichi Kawarabayashi and Y.~Kobayashi.
\newblock {Linear min-max relation between the treewidth of H-minor-free graphs
  and its largest grid}.
\newblock In C.~D{\"u}rr and T.~Wilke, editors, {\em 29th International
  Symposium on Theoretical Aspects of Computer Science (STACS 2012)}, volume~14
  of {\em Leibniz International Proceedings in Informatics (LIPIcs)}, pages
  278--289, Dagstuhl, Germany, 2012. Schloss Dagstuhl--Leibniz-Zentrum fuer
  Informatik.

\bibitem{Klo94}
T.~Kloks.
\newblock {\em Treewidth. Computations and Approximations}, volume 842 of {\em
  LNCS}.
\newblock Springer, 1994.

\bibitem{LeafS12sube}
A.~Leaf and P.~Seymour.
\newblock Treewidth and planar minors.
\newblock Manuscript, 2012.

\bibitem{Morgenstern1994exis}
M.~Morgenstern.
\newblock Existence and explicit constructions of $q + 1$ regular ramanujan
  graphs for every prime power q.
\newblock {\em Journal of Combinatorial Theory, Series B}, 62(1):44 -- 62,
  1994.

\bibitem{Pro89}
A.~Proskurowski.
\newblock {\em Maximal Graphs of Path-width K Or Searching a Partial
  K-caterpillar}.
\newblock University of Oregon. Dept. of Computer and Information Science,
  1989.

\bibitem{RobertsonS-II}
N.~Robertson and P.~D. Seymour.
\newblock Graph minors. {II}. algorithmic aspects of tree-width.
\newblock {\em Journal of Algorithms}, 7:309--322, 1986.

\bibitem{RobertsonS86GMV}
N.~Robertson and P.~D. Seymour.
\newblock {Graph minors. V. Excluding a planar graph}.
\newblock {\em J. Combin. Theory Series B}, 41(2):92--114, 1986.

\bibitem{RobertsonST94}
N.~Robertson, P.~D. Seymour, and R.~Thomas.
\newblock Quickly excluding a planar graph.
\newblock {\em J. Combin. Theory Ser. B}, 62(2):323--348, 1994.

\end{thebibliography}

\end{document}